\documentclass[11pt]{amsart}  

\usepackage[utf8]{inputenc}
\usepackage{amsmath}
\usepackage{amsfonts}
\usepackage{amssymb}
\usepackage{latexsym}
\usepackage{graphicx}
\usepackage{subfigure}
\usepackage{color}
\usepackage{mathtools}
\usepackage[backref=page]{hyperref}
\usepackage{verbatim}
\usepackage[all]{xy}
\usepackage{pdfsync}
\usepackage{youngtab}
\usepackage{ytableau}
\usepackage{young}
\usepackage{tikz}
\usepackage{cancel} 
\usepackage[normalem]{ulem} 
\usepackage{graphicx}
\usepackage{forest}
\usepackage{array}
\usepackage{extarrows}

\theoremstyle{plain}
\newtheorem{theorem}{Theorem}[section]
\newtheorem{proposition}[theorem]{Proposition}

\newtheorem{lemma}[theorem]{Lemma}
\newtheorem{corollary}[theorem]{Corollary}

\newtheorem{example}[theorem]{Example}

\newtheorem{remark}[theorem]{Remark}

\usepackage{todonotes}
\numberwithin{equation}{section}
\newcommand{\bN}{\mathbb{N}}

\newcommand{\bR}{\mathbb{R}}
\newcommand{\bZ}{\mathbb{Z}}

\newcommand{\suchthat}{\;|\;}

\newcommand{\p}[1]{\mathcal #1} 

\definecolor{emphcol}{rgb}{0.0, 0.5, 2.0}
\newcommand{\bemph}[1]{\textcolor{blue}{\emph{#1}}} 

\DeclareMathOperator{\GCD}{GCD}

\DeclareMathOperator{\md}{mod}

\newcommand{\shift}{\mathsf{shift}} 
\newcommand{\sort}{\mathsf{sort}} 

\newcommand{\breakd}{\mathsf{Break}} 
\newcommand{\park}{\mathsf{Park}} 

\begin{document}

\title[]{Divisors on complete multigraphs
and Donaldson-Thomas invariants of loop quivers}
\author{Matja\v z Konvalinka}
\address{Department of Mathematics, University of Ljubljana \& Institute of Mathematics, Physics and Mechanics, Ljubljana, Slovenia}
\thanks{The first author acknowledges the financial support from the Slovenian Research Agency (research core funding No. P1-0294).}
\email{\href{mailto:matjaz.konvalinka@fmf.uni-lj.si}{matjaz.konvalinka@fmf.uni-lj.si}}
\author{Markus Reineke}
\address{Faculty of Mathematics, Ruhr University Bochum, Bochum, Germany}
\email{\href{mailto:markus.reineke@rub.de}{markus.reineke@rub.de}}
\author{Vasu Tewari}
\address{Department of Mathematics, University of Hawaii at Manoa, Honolulu, HI 96822, USA}
\email{\href{mailto:vvtewari@math.hawaii.edu}{vvtewari@math.hawaii.edu}}

\begin{abstract}
We study the action of $S_n$ on the set of break divisors on complete multigraphs $K_{n}^m$.
We provide an alternative characterization for these divisors, by virtue of which we show that orbits of this action are enumerated by  the numerical Donaldson-Thomas invariants of $(m+1)$-loop quivers. Our characterization also allows us to restrict this action to $S_{n-1}$ and we identify the resulting $S_{n-1}$-module as that afforded by $K_n^m$-parking functions.
\end{abstract}

\maketitle

\section{Introduction}

Fix positive integers $m$ and $n$.
In \cite{KT21,KST21}, the first and third author introduced a family of $S_{mn}$ modules of cardinality $n^{mn-2}$ with the property that in the case $m=1$, the resulting representation restricted to $S_{n-1}$ is Haiman's well-known parking function representation \cite{Hai94}.
The primary motivation for introducing these modules was to gain a deeper understanding of the work of Berget-Rhoades \cite{BR14}, which studies certain $S_n$-modules with dimension $n^{n-2}$ that also restrict to the parking function representation.
In fact, Berget and Rhoades have a more general family of $S_n$-modules with dimension $m^{n-1}n^{n-2}$, which have the property that they carry actions of both $S_{n}$ and $S_{n-1}$. An explicit decomposition into irreducibles for these more general modules is determined for the $S_{n-1}$-action (but not the $S_{n}$-action) in \cite[Theorem 7]{BR14}.
We remark here that, unlike in the $m=1$ case \cite[Theorem 2]{BR14}, the case of general $m$  does not explicitly mention any analogue of parking functions.

Fix $G=K_n^m$, the complete multigraph on $n$ vertices with exactly $m$ edges between any two distinct vertices.
Denote the set of break divisors on $G$ by $\breakd_{m,n}$, and the set of $G$-parking functions (essentially $q$-reduced divisors for some distinguished vertex $q$) by $\park_{m,n}$.
The former naturally carries an $S_n$ action while the latter carries an $S_{n-1}$ action.
The aim of this note, achieved in Theorem~\ref{thm:frob_permutahedron}, is two-fold.

First we `amend' the modules in \cite{KST21} so that we obtain $S_n$-modules $\widehat{\p D}_{m,n}$ of dimension $m^{n-1}n^{n-2}$. These modules are then shown to be $S_n$-isomorphic to the module determined by $\breakd_{m,n}$, and furthermore allow us to show that their restriction to  $S_{n-1}$ is isomorphic to the $S_{n-1}$-action on $\park_{m,n}$. This generalizes our main result in \cite{KST21}.

Second, by exploiting the isomorphism $\widehat{\p D}_{m,n}\cong_{S_n} \breakd_{m,n}$, we show that the number of $S_n$-orbits on $\breakd_{m,n}$ equals the \emph{(unquantized) Donaldson-Thomas invariants} $\mathrm{DT}_n^{m+1}$ of the $(m+1)$-loop quiver; see \cite{Rei12} for more on combinatorial and other aspects. Very briefly, Donaldson-Thomas invariants of quivers (with potential) were introduced in \cite{KS11} as a mathematical definition of the string-theoretic concept of BPS state count; they are defined formally via Euler product factorizations of motivic generating series. Realizing the latter as Poincar\'e series of so-called Cohomological Hall algebras, integrality and positivity of Donaldson-Thomas invariants of symmetric quivers were established in \cite{Efi12}.

In the particular example of the $(m+1)$-loop quiver (and zero potential), the Donaldson-Thomas invariants ${\rm DT}_n^{m+1}$ can be defined concisely by factoring the generating series of $(m+1)$-ary trees with $n$ nodes
$$F(t)=\sum_{n\geq 0}\frac{1}{mn+1}{{(m+1)n}\choose{n}}t^n$$
into a (signed) Euler product:
$$F(t)=\prod_{n\geq 1}(1-((-1)^{m}t)^n)^{-(-1)^{mn}n{\rm DT}_n^{m+1}}.$$
Thus, from our main result, we obtain another combinatorial proof of the integrality of these numbers, and it is worthwhile to compare it with the earlier interpretation obtained by the second author.

In \cite[\S 6]{Rei12}, the natural cyclic action of $\mathbb{Z}_n\coloneqq \mathbb{Z}/n\mathbb{Z}$ on lattice points of the $mn$-fold dilation of the standard simplex in $\mathbb{R}^n$ is used to obtain a combinatorial interpretation for the $\mathrm{DT}_n^{m+1}$. These numbers, in fact quantized analogues thereof, are shown to count \emph{primitive/nearly-primitive elements} under this cyclic action, and the parity of $m$ plays a role.
In contrast, we utilize an $S_n\times \mathbb{Z}_n$ action on lattice points in a disjoint union of certain slices of the cube $[0,mn-1]^n$.
Identifying lattice points that are in the same $\mathbb{Z}_n$ class then gives an $S_n$-module $\widehat{\p D}_{m,n}$. In fact, each such class contains a unique lattice point belonging to a usual permutahedron closely related to the $m$-fold dilation of the standard permutahedron in $\bR^n$.
Finally we note that (part of) Theorem~\ref{thm:frob_permutahedron} may be interpreted as saying that the invariants $\mathrm{DT}^{m+1}_n$ for $(m+1)$-loop quivers equal the dimension of the space of $S_n$-invariants of $\widehat{\p D}_{m,n}$.   
In a similar vein (and in the general setting of symmetric quivers), Efimov \cite{Efi12} interprets the quantized DT-invariants as dimensions of spaces of $S_n$-invariants in certain quotients, but his work does not offer an explicit combinatorial perspective. He raises the question of exploring the underlying combinatorial aspects in \cite[\S 4]{Efi12}. The results in this article suggest looking for a graded analogue of $\widehat{\p D}_{m,n}$ that is $S_n$-isomorphic to Efimov's modules, thereby providing a tantalizing link to Cohomological Hall algebras. This is work in progress.

\section{Break divisors, $q$-reduced divisors, and symmetric group actions}
\label{sec:setup}
We fix positive integers $m$ and $n$ throughout.
By $[n]$ we mean $\{1,\dots,n\}$.
For all undefined terminology in the context of symmetric functions and symmetric group representations, we refer the reader to \cite[Chapter 7]{St99}. Throughout, given a $G$-set $X$ we refer to both the set and the corresponding $\mathbb{C}G$-module by $X$. We denote by $\mathrm{Frob}$ the \bemph{Frobenius characteristic} map assigning to the irreducible Specht module $V^{\lambda}$ indexed by a partition $\lambda$ the Schur function $s_{\lambda}$. We can extend $\mathrm{Frob}$ linearly and  compute the image of any $S_n$-module $V$ by  decomposing it into irreducibles.

\subsection{Break divisors on connected graphs}
\label{subsec:break divisors}
Given a finite graph $G$ (with multiple edges between the same vertices allowed), we denote its sets of vertices and edges by $V(G)$ and $E(G)$ respectively.
The \bemph{genus} $g(G)$ of a connected graph $G$ is defined to be $|E(G)|-|V(G)|+1$.

We now briefly recall some notions from Baker-Norine's theory \cite{BN07}.
A map $D:V(G)\to \bZ$ is called a \bemph{divisor}.  We say that $D$ is \bemph{effective} if $D(v) \geq 0$ for all $v\in V(G)$.  The \bemph{degree} $\deg(D)$ of $D$ equals $\sum_{v\in V(G)}D(v)$.
We write divisors either as tuples, say after identifying $V(G)$ with the set $[|V(G)|]$, or as formal sums $D = \sum_{v \in V(G)} D(v) (v)$.

For any orientation $\p O$ of the edges of $G$, define the divisor $D_{\p O}$ by
\begin{align}
D_{\p O}=\sum_{v\in V(G)}(\mathrm{indeg}_{\p O}(v)-1)(v).
\end{align}
Such divisors are called \bemph{orientable}.
Given $q\in V(G)$, we say that $\p O$ is $q$-connected if there exists a directed path from $q$ to any other vertex in $G$.
A \bemph{$q$-orientable divisor} is a divisor of the form $D_{\p O}$ where $\p O$ is $q$-connected.
A \bemph{break divisor} \cite{MZ08,ABKS14} on $G$ is an effective divisor $D$ of degree $g(G)$ such that  for all induced subgraphs $H$ of $G$ the following holds:
\begin{align}
\label{ineq:condition for break}
	\deg(D|_{H})\geq g(H).
\end{align}
Here $\deg(D|_{H})$ denotes the degree of $D$ restricted to vertices in $H$.
We denote the set of break divisors on $G$ by $\mathrm{Break}(G)$.

We record a result next that we have not been able to locate in the literature, though undoubtedly it should be well known to experts.
 Let $e_1,\dots,e_n$ denote the standard basis vectors in $\bR^n$.
  Let $G$ be a connected multigraph with $V(G)=[n]$.
 Then $G$ determines a zonotope $\p Z_G$ called the \bemph{graphical zonotope} obtained  by taking the Minkowski sum of line segments $[e_i,e_j]$, one for each edge $\{i,j\}\in E(G)$.
 Suppose that $\Delta_{n-1,n}$ denotes the $(n-1)$th standard hypersimplex in $\bR^n$ obtained by taking the convex hull of the $S_n$ orbit of the point $(1^{n-1},0)$.
 We define the  Minkowski difference $P-Q$ of polytopes $P,Q\subset \bR^n$ to be $\{x\in \bR^n\suchthat x+q\in P\text{ for all }q\in Q\}$.
 We then have the following result.
 \begin{proposition}
 \label{prop:break and  lattice points}
 For any connected multigraph $G$ with $V(G)=\{1,\dots,n\}$ we have
\[
\mathrm{Break}(G)=(\p Z_G-\Delta_{n-1,n})\cap \bZ^n.
\]
 \end{proposition}
 \begin{proof}
 The proof that follows was outlined to us by Chi Ho Yuen.
 Pick an orientation $\p O$ of the edges of $G$.
 Let $\mathrm{indeg}_{{\p O}}(v)$ denote the number of edges directed into the vertex $v$.
 The map $\p O \mapsto (\mathrm{indeg}_{{\p O}}(v))_{v\in V(G)}$ sets up a surjection between orientations on $G$ and lattice points in $\p Z_G$.
 It follows that the lattice points in $\p Z_G -(1^n)$ are precisely the orientable divisors on $G$.
 Thus, to establish the claim it suffices to show that
 \begin{align}
 \label{eq:reinterpreted_claim}
 D\text{ is a break divisor} \Longleftrightarrow D-(q) \text{ is orientable } \forall q\in V(G).
 \end{align}

 First assume that $D$ is a break divisor. Then \cite[Lemma 3.3]{ABKS14} tells us that $D-(q)$ is  $q$-orientable for any $q\in V(G)$. Thus the forward direction is established.

 Now assume that $D$ is a divisor such that $D-(q)$ is orientable for all $q\in V(G)$. We claim that $D-(q)$ is in fact $q$-orientable for all $q\in V(G)$. Having shown this, it will follow from \cite[Lemma 3.3]{ABKS14} that $D$ is a break divisor.

 Given $S\subset V(G)$, let $G[S]$ denote the subgraph of $G$ induced by $S$.
 Let $\chi(S)$ denote the topological Euler characteristic of $G[S]$, i.e. $\chi(S)=V(G[S])-E(G[S])$.
  Given any divisor $D$ define
  \begin{align}
  \chi(S,D)=\deg(D|_S)+\chi(S).
  \end{align}
  Fix $q\in V(G)$.
 Since $D-(q)$ is orientable, by \cite[Theorem 4.8]{ABKS14} we know that
 \begin{align}
 \chi(S,D-(q)) \geq 0
 \end{align}
 for every nonempty subset $S\subset V(G)$.
 If $D-(q)$ is not $q$-orientable, then \cite[Lemma 4.11]{ABKS14} tells us that
 \begin{align}
 \chi(S,D-(q)) \leq 0
 \end{align}
 for some nonempty subset $S\subset V(G)\setminus\{q\}$.
 Thus it must be the case that
 \begin{align}
 \label{eq:to be compared}
 \chi(S,D-(q))=\deg(D-(q)|_S)+\chi(S)=0
 \end{align}
 for some nonempty subset $S\subset V(G)\setminus \{q\}$.

 For such an $S$, pick any $p\in S$ and consider $\chi(S,D-(p))$. Since $p\in S$, we have that
 \begin{align}
 \deg(D-(p)|_S)<\deg(D-(q)|_S).
 \end{align}
 On comparing with \eqref{eq:to be compared}, it follows that $\chi(S,D-(p))<0$. By \cite[Theorem 4.8]{ABKS14}, this contradicts our assumption that $D-(p)$ is orientable. It thus follows that $D-(q)$ is $q$-orientable, which concludes the proof.
 \end{proof}
The above proof gives yet another perspective on the following result; see \cite{Yu17} for more fascinating insights on this matter.
\begin{corollary}
For a connected multigraph $G$, the number of break divisors on $G$ is the number of spanning trees.
\end{corollary}
\begin{proof}
By \cite[Corollary 11.5]{Pos09} we know that $\p Z_G-\Delta_{n-1,n}$ has as many lattice points as the volume of $\p Z_G$, and the latter is well known to equal the number of spanning trees of $G$.
\end{proof}
\begin{remark}\emph{
Of course, one could take for granted the fact that the number of break divisors equals the number of spanning trees, and then use Proposition~\ref{prop:break and  lattice points} to prove Postnikov's result \cite[Corollary 11.5]{Pos09}. This approach is quite different from that in \emph{loc. cit.}, which relies on mixed subdivisions and the Cayley trick. }
\end{remark}

\subsection{The case of the complete multigraph}
Recall that  $K_{n}^m$ is the graph on the vertex set $[n]$ with $m$ edges between vertices $i$ and $j$ for all $1\leq i<j\leq n$.  Its genus is given by
\begin{align}
\label{eq:def_genus_knm}
g_{m,n}\coloneqq g(K_n^{m})=m\binom{n}{2}-n+1.
\end{align}
Our focus henceforth is primarily on $K_n^m$.
 Let $\breakd_{m,n}$ denote the set of break divisors on $K_n^m$.
 Let ${\p P}_{m,n}$ be the permutahedron in $\bR^n$ obtained as the convex hull of the $S_n$-orbit of $(m(n-1)-1,m(n-2)-1,\dots,m-1,0)$.
 In the case $m=1$, this permutahedron is exactly the \emph{trimmed permutahedron} that plays a key role in \cite{KST21}.

 \begin{lemma}
 	We have
	\[
	\breakd_{m,n}=\p P_{m,n}\cap \bZ^n.
	\]
 \end{lemma}
 \begin{proof}
 We give two arguments. For the first note that the zonotope $\p Z_{K_n^m}$ is given by the $m$-fold dilation of the \emph{standard permutahedron}, i.e. its vertices are given by the $S_n$-orbit of $m\cdot(n-1,n-2,\dots,1,0)$.
 It follows that $\p Z_{K_n^m}-\Delta_{n-1,n}$ is indeed $\p P_{m,n}$. The claim now follows from Proposition~\ref{prop:break and  lattice points}.

 Alternatively, identifying $V(K_n^m)$ with $[n]$ as usual,
 let $D=(d_1,\dots,d_n)\in \breakd_{m,n}$. Then $\sum_{1\leq i\leq n}d_i=g_{m,n}$.
 Since any induced connected subgraph $H$ of $K_{n}^m$ is isomorphic to $K_{j}^m$ for some positive integer $j$, the condition in \eqref{ineq:condition for break}  translates to
 \begin{align}
 	\sum_{i\in S}d_i \geq g_{m,|S|}=m\binom{|S|}{2}-|S|+1.
 \end{align}
 for every nonempty $S\subset [n]$. These inequalities define the  permutahedron ${\p P}_{m,n}$; see \cite{Rad52}.
 \end{proof}
 Since $K_n^m$ has $m^{n-1}n^{n-2}$ spanning trees, we infer that
 \begin{align}
 	|{\p P}_{m,n}\cap \bZ^n|= m^{n-1}n^{n-2}.
 \end{align}
 \begin{example}
 \emph{
 Let $m=2$ and $n=3$. Then $\breakd_{2,3}$ contains 12 elements: the six permutations of $(3,1,0)$, as well as the three permutations each of $(2,2,0)$ and $(2,1,1)$. Note that these elements are exactly the lattice points in the permutahedron ${\p P}_{2,3}$.}
 \end{example}

We now recall another notion of interest.
Fix $q\in V(G)$.
A \bemph{$q$-reduced divisor} \cite[\S~3.1]{BN07} is a divisor $D$ such that $D(v)\geq 0$ for $v\in V(G)\setminus\{q\}$, and additionally, for every nonempty  $S\subset V(G)\setminus \{q\}$ there exists $v\in S$ satisfying $D(v)<\mathrm{outdeg}_S(v)$.\footnote{The chip-firing perspective is helpful here. This condition says that if all vertices in $S$ fire simultaneously, at least one of them will be in debt.}
Here $\mathrm{outdeg}_S(v)$ is the number of edges in $G$ connecting  $v$ to vertices in $V(G)\setminus S$. 
Since the quantity $D(q)$ does not play any role in these inequalities, one can ignore it.
The function $D$ restricted to $V(G)\setminus \{q\}$ is exactly what is known as a \bemph{$G$-parking function} \cite{Pos04}.

We immediately specialize to the case $G=K_n^m$ with vertex set $[n]$, and set $q=n$.
By a result of Hopkins-Gaydarov \cite[Theorem 2.5]{GH16}, fortuitously, the set of $K_m^n$-parking functions may be  characterized as a set of vector parking functions for an appropriate vector. This is also easy to observe from the characterization of $q$-reduced divisors in the preceding paragraph.
 Indeed, the sequence $(d_1,\dots,d_{n-1})$ is a $K_{n}^m$-parking function if its weakly increasing  rearrangement $(\tilde{d}_1,\dots,\tilde{d}_{n-1})$ satisfies
\begin{align}
\tilde{d}_i\leq mi-1.
\end{align}
In other words, $(d_1,\dots,d_{n-1})$ is a $K_{n}^m$-parking function if and only if there are at least $i$ entries $\leq mi-1$ for $1\leq i\leq n-1$.
Denote the set of $K_n^{m}$-parking functions by $\park_{m,n}$.
When $m=1$ this immediately reduces to the definition of classical parking functions.

There is a notion of linear equivalence on divisors defined in \cite{BN07}.
For any connected graph $G$, it turns out that each linear equivalence class of degree $g(G)$ divisors contains a unique break divisor.
Furthermore, every linear equivalence class contains a unique $q$-reduced divisor \cite[Proposition 3.1]{BN07}, and so we infer that
\[
|\breakd_{m,n}|=|\park_{m,n}|=m^{n-1}n^{n-2}.
\]
The characterization of $q$-reduced divisors as vector parking functions implies that $\park_{m,n}$ carries a permutation action of $S_{n-1}$.
That break divisors are lattice points in a certain permutahedron implies that $\breakd_{m,n}$ carries a permutation action of $S_n$.
{\sf Is there a relation between the resulting modules?}

To motivate our main theorem we consider an example.
\begin{example}
\label{ex:demo_main}
\emph{
Consider $m=2$ and $n=3$.
On the one hand, the Frobenius characteristic of the $S_3$ action on $\breakd_{2,3}$ equals
 \[
 \mathrm{Frob}(\breakd_{2,3})=h_{111}+2h_{21}=3s_3+4s_{21}+s_{111},
 \]
 where $h$ denotes the complete homogeneous symmetric function.
}
\emph{On the other hand, the set $\park_{2,3}$ contains 12 elements: ordered pairs $(d_1,d_2)$ where $0\leq d_1\leq 1$ and $0\leq d_2\leq 3$, and their rearrangements. The Frobenius characteristic of the $S_2$ action on $\park_{2,3}$ is
\[
\mathrm{Frob}(\mathrm{Park}_{2,3})=2h_2+5h_{11}=7s_2+5s_{11}.
\]
The reader may now verify that $\mathrm{Park}_{2,3}=\mathrm{Res}_{S_{2}}^{S_3}\breakd_{2,3}$ where $\mathrm{Res}$ denotes restriction.
As we shall see in Theorem~\ref{thm:frob_permutahedron}, this phenomenon is part of a larger picture.
}
\end{example}

Before we offer a unifying perspective on these two symmetric group representations, we make a remark.
It is true that one can see  the fact that $\mathrm{Park}_{m,n}=\mathrm{Res}_{S_{n-1}}^{S_n}\breakd_{m,n}$ from the fact that linear equivalence classes on divisors of degree $g(K_n^m)$ contain a unique break divisor as well as a unique $q$-reduced divisor.
That being said, gleaning any further information about the representation $\breakd_{m,n}$, say character values or the multiplicity of the trivial representation, is not immediate from the definition of break elements.
This opacity motivates the perspective we proceed to describe in what follows.

\section{The modules \texorpdfstring{$\p D_{m,n}$}{D m,n} and \texorpdfstring{$\widehat{\p D}_{m,n}$}{dD m,n} }

Set $N\coloneqq mn$.
Consider the set of $n$-tuples defined as follows:
\begin{align*}
\p D_{m,n}\coloneqq \{(x_1,\dots,x_{n})\suchthat 0\leq x_i\leq N-1, \sum_{1\leq i\leq n}x_i=g_{m,n} \: (\md N)\}.
\end{align*}
The cardinality of $\p D_{m,n}$ is clearly $N^{n-1}$.
The symmetric group $S_n$ acts by permutations and the orbits are indexed by partitions that fit in an $n\times (N-1)$ box and have size congruent to $g_{m,n}$ modulo $N$, or equivalently, multisets of size $n$ with entries drawn from $\{0,\dots,N-1\}$ and summing to $g_{m,n}$ modulo $N$.
Let us denote the set of $S_n$-orbits by $S_n\backslash {\p D}_{m,n}$.
We first compute the cardinality of this set, as we will need it subsequently.

To this end, we recall a result of von Sterneck from the early 1900s; see \cite[Theorem 3]{Ra44} for a statement in English. Given positive integers $a$ and $b$, consider the \bemph{Ramanujan sum} \cite[Theorem 272]{HW08}
\begin{align}
\label{eq:ramanujan sum}
C_b(a)\coloneqq \sum_{\substack{1\leq k\leq b\\ \gcd(k,b)=1}}\mathrm{exp}(2\pi ika/b)=\mu\left(\frac{b}{\gcd(a,b)}\right)\frac{\phi(b)}{\phi\left(\frac{b}{\gcd(a,b)}\right)},
\end{align}
where $\mu$ is the number-theoretic M\"{o}bius function and $\phi$ is the Euler phi function.
\begin{lemma}
The number of multisets of size cardinality $k$ with entries drawn from $\{0,\dots, a-1\}$ with subset sum $b\!\!\mod a$ equals
\[
\frac{1}{a}\sum_{d|a,k}\binom{\frac{a+k}{d}-1}{\frac{k}{d}}C_d(b).
\]
\end{lemma}
We are interested in the case $a=N$, $k=n$, and $b=g_{m,n}$.
We thus have
\begin{align}
|S_n\backslash {\p D}_{m,n}|=\frac{1}{mn}\sum_{d|n}\binom{\frac{(m+1)n}{d}-1}{\frac{n}{d}}C_d(g_{m,n}).
\end{align}
It remains to substitute the Ramanujan sum $C_d(g_{m,n})$ for $d|n$. Equation~\eqref{eq:ramanujan sum} tells us that
\begin{align}
C_d(g_{m,n})=\mu(d/r)\frac{\phi(d)}{\phi(d/r)},
\end{align}
where $r=\gcd(d,g_{m,n})=\gcd(d,m\binom{n}{2}-n+1)$. If $m$ is even or $n$ is odd, we must have $r=1$.
This leaves the case $m$ being odd and $n$ being even. Taking cases based on $n\!\mod 4$ we find that if $n\!\mod 4=0$, then  $r=1$.  Otherwise
\[
r=\left\lbrace\begin{array}{ll}1 & d \text{ odd,}\\ 2 & d \text{ even.}\end{array}\right.
\]
In summary, our calculations above imply that if $m$ is odd and $n\equiv 2\!\!\mod 4$, then
\begin{align}
|S_n\backslash \p D_{m,n}|=
\frac{1}{mn}\displaystyle\sum_{\substack{d|n\\ d \text{ odd}}}\mu(d)\binom{\frac{(m+1)n}{d}-1}{\frac{n}{d}}-\frac{1}{mn}\displaystyle\sum_{\substack{d|n\\ d \text{ even}}}\mu(d)\binom{\frac{(m+1)n}{d}-1}{\frac{n}{d}}.
\end{align}
In all other cases we get
\begin{align}
|S_n\backslash \p D_{m,n}|=
\frac{1}{mn}\displaystyle\sum_{d|n}\mu(d)\binom{\frac{(m+1)n}{d}-1}{\frac{n}{d}}.
\end{align}
Rewriting $d$ as $n/d$ and modifying appropriately, we record the preceding computation as
\begin{proposition}
\label{prop:orbits of D}
The number of orbits of $\p D_{m,n}$ under the permutation action of $	S_n$ is
\begin{align*}
|S_n\backslash \p D_{m,n}|=\frac{1}{n}\sum_{d|n} (-1)^{m(n+d)}\mu(n/d)\binom{(m+1)d-1}{md}.
\end{align*}
\end{proposition}
Up until this point we have only considered the $S_n$ action on ${\p D}_{m,n}$. There is a $\mathbb{Z}_n$ action on this set as well. It does not arise from cyclic rotation of coordinates and instead considers `translation' by the vector $m\cdot (1,\dots,1)$. We explore this action next.

\subsection{The module \texorpdfstring{$\widehat{\p D}_{m,n}$}{DHat m,n}}

We  define a simple cyclic action on $\p D_{m,n}$ which groups its elements into $m^{n-1}n^{n-2}$ cyclic classes. This cyclic action commutes with the action of $S_n$, and thus the cyclic classes end up inheriting an $S_n$ action too.

Define the \bemph{shift} map $\shift$ mapping $\{0,\dots,N-1\}^n$ to itself via
\begin{align*}
\shift(x_1,\dots,x_n)\coloneqq (x_1+m,\dots,x_n+m),
\end{align*}
where addition is performed modulo $N$.
This gives an equivalence relation $\sim$ on $\p D_{m,n}$: two sequences are \emph{shift-equivalent} if one is obtained by applying $\shift^j$ to the other  for some $j\in \bN$.
As mentioned above, $S_n$ acts on $\widehat{\p D}_{m,n}\coloneqq \p D_{m,n}/\sim$.

In addition to having the right dimension, $\widehat{\p D}_{m,n}$ turns out to possess an additional desirable property:  every equivalence class in $\p D_{m,n}/\!\sim$  contains a unique break divisor in $\breakd_{m,n}$ and, assuming we drop the last coordinate, a unique $G$-parking function in $\park_{m,n}$.
Indeed, going back to Example~\ref{ex:demo_main}, note for instance that the shift-equivalence class of $(2,2,0)\in \breakd_{2,3}$ is $\{(2,2,0),(4,4,2),(0,0,2)\}$.
Omitting the last coordinates, we see that $(0,0)$ is the unique element in $\park_{2,3}$. Let us henceforth, given $\mathbf{a}=(a_1,\dots,a_{n-1})$, denote by $\pi(\mathbf{a})$ the sequence $(a_1,\dots,a_{n-1})$ obtained by omitting the last coordinate.
We now proceed toward establishing the claim in general.

\begin{proposition}
\label{prop:unique parking function}
Given a shift equivalence class $\mathcal{C}$ in $\widehat{\p D}_{m,n}$, there exists a unique element $\mathbf{a}\in {\p C}$ such that $\pi(\mathbf{a})\in \park_{m,n}$.
\end{proposition}
\begin{proof}
We adapt a folklore argument for counting classical parking functions attributed to Pollak.

Fix $\mathbf{b}\in {\p C}$.
Consider $N$ parking spots labeled $0$ through $N-1$ arranged clockwise along a circle, with $N-1$ neighboring $0$.
Consider $n-1$ cars labeled $1$ through $n-1$, and let $b_i$ denote the preferred parking spot for car $i$.
The `usual' rules of parking apply: the cars come in the order $1$ through $n-1$, each car takes its preferred spot if it is free, otherwise continues clockwise and parks at the next free spot.

Let $O_{\mathbf{b}}$ denote the set of occupied spots after all cars have parked.
We consider the parking spots as $n$ contiguous blocks of size $m$ each. Define the sequence $c=(c_1,\dots,c_n)$ by setting
\[
c_i=O_{\mathbf{b}}\cap \{m(i-1),\dots,mi-1\}.
\]
Clearly $c_1+\cdots+c_n=n-1$. A routine application of the cycle lemma implies that there is a unique rotation $\tilde{c}=(c_{j+1},\dots,c_n,c_1,\dots,c_j)$ of $c$ with the property that
\[
\tilde{c}_1+\dots+\tilde{c}_k \geq k
\]
for all $1\leq k\leq n-1$. This in turn means that $\tilde{b}=\shift^{n-j}(\mathbf{b})$ has the property that $\pi(\tilde{b})\in\park_{m,n}$, and that $j$ is the only choice with this property.
 \end{proof}
Note that the proof makes no use of the last coordinate of elements in ${\p D}_{m,n}$.
\begin{example}
\emph{
Consider $m=3$ and $n=5$. Pick $\mathbf{b}=(3,13,7,13,5)\in {\p D}_{3,5}$. The remaining elements in the shift-equivalence class of $\mathbf{b}$ are
\[
\{(6,1,10,1,8),(9,4,13,4,11),(12,7,1,7,14),(0,10,4,10,2)\}.
\]
When we park cars in the spots given by $\pi(\mathbf{b})=(3,13,7,13)$, the occupied spots are given by $\{3,7,13,14\}$. Thus the sequence $c=(c_1,\dots,c_5)$ is given by $(0,1,1,0,2)$. The rotated version $\tilde{c}$ with the desired property is $(c_5,c_1,\dots,c_4)=(2,0,1,1,0)$.
Now consider $\shift^{5-4}(\mathbf{b})=(6,1,10,1,8)$. It is easily checked that $(6,1,10,1)\in \park_{3,5}$.}
\end{example}

Next we show that every shift-equivalence class in $\widehat{\p D}_{m,n}$ contains a break divisor.
We will need a preliminary lemma.
Let $\Lambda_n$ denote the set of partitions $\lambda=(\lambda_1\geq \cdots \geq \lambda_n\geq 0)$ in $\bN^n$.
Given $\mathbf{x}=(x_1,\dots,x_n)\in \bN^n$, define $\sort(\mathbf{x})$ to be the partition obtained by sorting $\mathbf{x}$ in nonincreasing order.

\begin{lemma}\label{lem:unique_rep_dominated}
Given $\lambda=(\lambda_1,\dots,\lambda_n)\in \Lambda_n\cap \breakd_{m,n}$, no element in the set $\{\sort\circ \shift^j(\lambda)\suchthat 1\leq j\leq n-1\}$ belongs to $\Lambda_n\cap \breakd_{m,n}$.
\end{lemma}
\begin{proof}
The argument is entirely similar to \cite[Lemma 4.2]{KST21}; one simply needs to incorporate the parameter $m$ carefully. For the sake of completeness, we give the argument.

Any $\lambda$ sitting (in French notation) inside an $n\times N$ box can be viewed  as  a lattice path $L_{\lambda}$ going from $(N,0)$ to $(0,n)$.
We extend this to an bi-infinite path $L_{\lambda}^{\infty}$ by repetition.
Label the horizontal steps in $L_{\lambda}$ with integers $0$ through $N-1$ going right to left.
Fix a $j$ such that $0\leq j\leq n-1$ and  consider the fragment $L'$ of $L_{\lambda}^{\infty}$ of length $(m+1)n$  that starts with the horizontal step labeled $mj$ and proceeds northwest.

Observe that  $L'$ determines the partition $\sort\circ \shift^j(\lambda)$ when viewed in the $n\times N$ box it lives in.
If we let $i$ denote the number of vertical steps in $L_{\lambda}$ preceding the horizontal step labeled $mj$, then we have
\begin{align}\label{eqn:size_change_upon_shifting}
	|\sort\circ \shift^j(\lambda)|=|\lambda|+(j-i)N.
\end{align}

\medskip

Now suppose there exists $1\leq j\leq n-1$ such that $\sort\circ \shift^j(\lambda)\in\Lambda_n\cap \breakd_{m,n}$.
By \eqref{eqn:size_change_upon_shifting} we must have $j=i$.
Thus, the horizontal step labeled $mj$ must touch the diagonal $x+my=N$.
Let $\nu=(\nu_1,\dots,\nu_{n-j})$ be the partition determined by the subpath of $L_{\lambda}$ restricted to the $(n-j)\times m(n-j)$ box in the top left.
Let $\mu=(\mu_1,\dots,\mu_j)$ be the partition determined by the subpath of $L_{\lambda}$ restricted to the $j\times mj$ box in the bottom right.

Observe that
\begin{align}
\lambda&=(m(n-j)+\mu_1,m(n-j)+\mu_2,\dots,m(n-j)+\mu_{j},\nu_1,\dots,\nu_{n-j})\\
\sort\circ \shift^j(\lambda)&=(mj+\nu_1,mj+\nu_2,\dots,mj+\nu_{n-j},\mu_1,\dots,\mu_j).
\end{align}

\medskip

Since $\lambda\in\Lambda_n\cap \breakd_{m,n}$, by definition it is dominated by the partition $\delta_{m,n}=(m(n-1)-1,m(n-2)-1,\dots,m\cdot 1-1,0)$.
By comparing the sum of the first $k$ parts of $\lambda$ with that of the first $k$ parts of $\delta_{m,n}$, we have
\begin{align}
\sum_{k=1}^{j}(m(n-j)+\mu_k) \leq (m(n-1)-1)+\cdots+ (m(n-j)-1).
\end{align}
Since the left-hand side is $|\lambda|-|\nu|=m\binom{n}{2}-n+1-|\nu|$, we may rewrite the above inequality as
\begin{align}\label{eqn:initial_ineq}
|\nu| \geq m\binom{n-j}{2}-(n-j)+1.
\end{align}

On the other hand, since our assumption is that $\sort\circ \shift^j(\lambda)$ is also dominated by $\delta_{m,n}$, by comparing the sum of the first $n-j$ parts we obtain
\begin{align}
\sum_{k=1}^{n-j}(mj+\nu_k)\leq (m(n-1)-1)+\cdots +(mj-1).
\end{align}
This in turn may be rewritten as
\begin{align}
\label{eqn:final_ineq}
	|\nu|\leq m\binom{n-j}{2}-(n-j),
\end{align}
which is in contradiction with the inequality in \eqref{eqn:initial_ineq}.
\end{proof}

\begin{proposition}
\label{prop:unique break divisor}
Given a shift equivalence class $\mathcal{C}$ in $\widehat{\p D}_{m,n}$, there exists a unique element $\mathbf{a}\in {\p C}$ such that $\mathbf{a}\in \breakd_{m,n}$.
\end{proposition}
\begin{proof}
Consider the map $\phi:\breakd_{m,n} \to \widehat{\p D}_{m,n}$ sending a break divisor $\mathbf{b}$ to the unique shift equivalence class $[\mathbf{b}]$ that contains it. Since $|\breakd_{m,n}|=|\widehat{\p D}_{m,n}|$, to prove the claim we only need to show that $\phi$ is an injection. This is straightforward given Lemma~\ref{lem:unique_rep_dominated}, and the details are the same as the proof of \cite[Theorem 4.3]{KST21}. So we omit them.
\end{proof}

We are now ready to state our main theorem connecting the various pieces.

\begin{theorem}\label{thm:frob_permutahedron}
	The representation $\breakd_{m,n}$  is isomorphic to the representation $\widehat{\p D}_{m,n}$.
Furthermore, the number of $S_n$-orbits on $\breakd_{m,n}$ equals the numerical DT-invariant $\mathrm{DT}^{m+1}_n$ for the $(m+1)$-loop quiver \cite{Rei11,Rei12}.
More precisely
\[
|S_n\backslash \widehat{\p D}_{m,n}|=
\frac{1}{n^2}\sum_{d|n} (-1)^{m(n+d)}\mu(n/d)\binom{(m+1)d-1}{md}.
\]
Finally, we have
	\[
\mathrm{Res}_{S_{n-1}}^{S_n}(\breakd_{m,n})=\park_{m,n}.
\]
\end{theorem}
\begin{proof}
  The claim $\breakd_{m,n}\cong_{S_n}\widehat{\p D}_{m,n}$  follows from the fact that the map $\phi$ in Proposition~\ref{prop:unique break divisor} is $S_n$-equivariant.

In arriving at Proposition~\ref{prop:orbits of D}, we accounted for the $S_n$-action. To count $S_n\times\mathbb{Z}_n$-orbits on ${\p D}_{m,n}$, we observe that each shift-equivalence class has size $n$, so we only need to divide the expression in Proposition~\ref{prop:orbits of D} by $n$. The fact that the resulting expression is a DT-invariant follows by comparing to \cite[Theorem 3.2]{Rei12}.\footnote{There is a minor typo in the expression in \emph{loc. cit.}; the $n$ in the binomial coefficient should be replaced by a $d$.}
This establishes the second claim.

  For the third claim, consider $S_{n-1}$ naturally as subgroup of  $S_n$ consisting of permutations that have $n$ as a fixed point.
  Let $(a_1,\dots,a_{n-1})\in \park_{m,n}$, and consider the unique element $\mathbf{a}\in {\p D}_{m,n}$ such that $\pi(\mathbf{a})=(a_1,\dots,a_{n-1})$. Then we know that there is a unique break divisor $\mathbf{b}$ shift-equivalent to $\mathbf{a}$. Now let $\sigma\in S_{n-1}$ act by permuting the first $n-1$ coordinates.
  Since shifts commute with the $S_n$-action on ${\p D}_{m,n}$, we have that $\sigma\cdot \mathbf{b}$ is shift-equivalent to $\sigma\cdot \mathbf{a}$.
  We conclude that $\mathrm{Res}_{S_{n-1}}^{S_n}(\breakd_{m,n})=\park_{m,n}$.
   \end{proof}

\begin{example}
\emph{
Consider $m=2$ and $n=4$. The orbits of the $S_n$ action on $\breakd_{m,n}$ are indexed by partitions $\lambda=(\lambda_1,\dots,\lambda_4)\vdash 9$ dominated by $(5,3,1,0)$. It is easily checked that there are $10$ such partitions. Let us check that this count matches the value of $\mathrm{DT}^{m+1}_{n}$.
The sum in Theorem~\ref{thm:frob_permutahedron} becomes
\[
\frac{1}{16}\left(\binom{11}{8}-\binom{5}{4}\right)=\frac{1}{16}\left(165-5\right)=10.
\]
}
\end{example}
We let $\chi_{m,n}$ denote the character of $\breakd_{m,n}$. 
The following corollary tells us that $\chi_{m,n}$  has a succinct description, which we do not know how to arrive at without the isomorphism $\breakd_{m,n}\cong_{S_n}\widehat{\p D}_{m,n}$.
The proof of this result follows the exact same route as laid out in the proof of \cite[Theorem 3.1]{KT21}, and we direct the reader to \emph{loc.\ cit.} for more details.

\begin{corollary}\label{cor:lattice_points_fixed}
Let $\sigma\in S_n$ have cycle type $\lambda=(\lambda_1,\dots,\lambda_{\ell})$ where $\lambda_{\ell}>0$.
Let $d\coloneqq \GCD(\lambda_1,\dots,\lambda_{\ell})$.
Then the number of break divisors in $\breakd_{m,n}$ fixed by $\sigma$, i.e. $\chi_{m,n}(\sigma)$, is given by
\begin{align*}
\chi_{m,n}(\sigma)=\left\lbrace\begin{array}{ll}
m^{\ell-1}n^{\ell-2} & d=1,\\
2m^{\ell-1}n^{\ell-2} & d=2, m \text{ odd, and } n=2 \: (\md 4),\\
0 & \text{otherwise.}
\end{array}\right.
\end{align*}
\end{corollary}

\section*{Acknowledgements}
V.T. is extremely grateful to Chi Ho Yuen for enlightening email correspondence.

\bibliographystyle{alpha}
\bibliography{Biblio_PS}

\end{document}